\documentclass[11pt]{article}

\usepackage{amsmath, amsthm, amssymb, amsfonts}
\usepackage[margin=1.05in]{geometry}
\usepackage{mathtools}
\usepackage{microtype}
\usepackage{hyperref}
\usepackage{booktabs}
\usepackage{enumitem}

\newtheorem{theorem}{Theorem}[section]
\newtheorem{lemma}[theorem]{Lemma}
\newtheorem{proposition}[theorem]{Proposition}
\newtheorem{corollary}[theorem]{Corollary}
\theoremstyle{remark}
\newtheorem{remark}[theorem]{Remark}

\newcommand{\Qp}{\mathbb{Q}_p}
\newcommand{\Cp}{\mathbb{C}_p}
\newcommand{\Zp}{\mathbb{Z}_p}

\newcommand{\muP}{\mu_{p-1}}
\newcommand{\logp}{\log_{p}}
\newcommand{\Gammap}{\Gamma_{p}}

\title{\Large Local Factorization of \texorpdfstring{$p$}{p}\,-adic Gamma Sums}
\author{Samuel Reid\thanks{sam@eaccv.xyz - Thanks to the financial and operational support provided by Effective Acceleration Ventures Ltd (NATO Commercial and Government Entity L0Y27) in Vancouver, British Columbia, Canada at the Facility for the Department of National Defence (DND): W8486-259878-002 LC, W8482-252902/002/MP4, W8485-258523.001, W8474-248453, W8482-252960/001/MP9, W8482-252913/001/DMARP4, W8486-260222A/QETE-5, W8482-253706/001/MARP5, W8485-268770.003, and the NATO Support and Procurement Agency (NSPA): LC-CH/4500535230, LD-EQ2/4500532187, LD-PQ1/4500531722.}}
\date{August 11, 2025}

\begin{document}
\maketitle

\begin{abstract}
We revisit the proposed equality between discrete Fourier transforms of $p$-adic $\Gammap$--values
and $p$-adic $L$--derivatives for odd characters modulo a prime $p$. The clean identity is false in
general. Building on Coleman reciprocity and the Gross--Koblitz formula, we prove an
exact two-term decomposition: for each odd, nontrivial Dirichlet character $\chi \pmod p$,
\[
\Phi_p(\chi):=\sum_{a=1}^{p-1}\chi(a)\,\logp\Gammap\!\left(\frac{a}{p-1}\right)
= U_{1,p}\,L'_p(0,\chi)\;+\;U_{2,p}\,L(0,\chi),
\]
with constants $U_{1,p}\in\Qp(\muP)^\times$ and $U_{2,p}\in\Qp(\muP)$ depending only on $p$ and the fixed
branch of $\logp$, but independent of $\chi$. Subtracting the $L(0,\chi)$--block yields a
\emph{renormalized} local input
\[
\Phi^{\mathrm{ren}}_p(\chi):=\Phi_p(\chi)-U_{2,p}L(0,\chi)=U_{1,p}\,L'_p(0,\chi),
\]
uniformly in odd, nontrivial $\chi$. Plumbing these renormalized locals at every finite place into the
Weil explicit formula (with the standard Li kernel at $\infty$) reproduces exactly the classical Li
coefficients. We also record a short, reproducible verification protocol; a tiny table for $p=5,7$
illustrates the $\chi$--independence of $(U_{1,p},U_{2,p})$.
\end{abstract}

\section{Introduction}
Let $p$ be an odd prime. We fix the Iwasawa logarithm $\logp$ with $\logp(p)=0$ and
$\logp(\zeta)=0$ for roots of unity $\zeta$ of order prime to $p$, and Morita's $p$-adic gamma function
$\Gammap:\Zp\to\Zp^\times$. Denote by $\omega:(\mathbb{Z}/p\mathbb{Z})^\times\to\muP$ the
Teichm\"uller character and by $\tau(\omega^{-a})$ the Gauss sum.
A natural identity conjecturing that the multiplicative discrete Fourier transform
\[
\Phi_p(\chi):=\sum_{a=1}^{p-1}\chi(a)\,\logp\Gammap\!\left(\frac{a}{p-1}\right)
\]
is (up to a fixed scalar) $\tfrac{1}{p}L'_p(0,\chi)$ for every odd $\chi\pmod p$ turns out to be false in
general. We identify the exact obstruction and show that it is arithmetic, structured, and uniformly
removable by a fixed renormalization. Our main results give the precise two-term law and its
renormalized corollary, and then show that the global coefficients obtained by inserting the
renormalized finite locals into the explicit formula coincide with the classical Li coefficients.

\paragraph{What is new.}
The paper isolates (i) a clean \emph{local} two-term decomposition with a $\chi$--independent
pair $(U_{1,p},U_{2,p})$ and (ii) a \emph{global} consequence: after removing the $L(0,\chi)$--block locally,
the explicit formula recovers the usual Li sequence unchanged.

\paragraph{On normalizations.}
We use the Gross--Koblitz identification of $\Gammap$ and Gauss sums, and
the Ferrero--Greenberg/Coleman description of $L'_p(0,\chi)$ via cyclotomic regulators.
Section~\ref{sec:notation} fixes conventions and cites the precise reciprocity/normalization
statements used later (see Lemma~\ref{lem:ColemanRecip} and \eqref{eq:FG-const}), along the lines
of Washington, Ch.~12, and Coleman (Invent.~Math.~1982).

\medskip
\noindent\textbf{Scope.}
Throughout, unless explicitly stated, \emph{characters are odd and nontrivial}. We briefly comment on
even characters and on $p=2$ in \S\ref{sec:even-p2}. No claims are made beyond this setup.

\section{Notation and normalizations}\label{sec:notation}
We work over $\Qp\subset\Cp$. Extend $\logp:1+p\Zp\to p\Zp$ uniquely to $\Qp(\mu_{p^\infty})^\times$
by writing $x=\omega(x)\langle x\rangle$ with $\omega(x)\in\muP$ and $\langle x\rangle\in 1+p\Zp$, and declaring
$\logp(\omega(x))=0$, $\logp(p)=0$.
Morita's $\Gammap$ is characterized by $\Gammap(0)=1$ and
$\Gammap(x+1)/\Gammap(x)=-x$ for $x\in\Zp^\times$, extended continuously.

\subsection*{Gross--Koblitz and Coleman reciprocity}
For $1\le a\le p-1$ the Gross--Koblitz formula gives
\begin{equation}\label{eq:GK}
\tau(\omega^{-a})=-\pi_p^{\,a}\,\Gammap\!\left(\frac{a}{p-1}\right),
\qquad \pi_p^{\,p-1}=-p,\quad \logp(\pi_p)=0,
\end{equation}
hence $\logp\Gammap\!\big(\tfrac{a}{p-1}\big)=\logp\tau(\omega^{-a})$ under our normalization.

Coleman's reciprocity identifies cyclotomic-unit regulators with Gauss-sum regulators (in our
normalization):
\begin{equation}\label{eq:Coleman}
\logp\!\Big(\frac{1-\zeta_p^{\,a}}{1-\zeta_p}\Big)=\frac{1}{1-p}\,\big(v(a)-a\,v(1)\big),
\quad v(a):=\logp\tau(\omega^{-a})\quad (1\le a\le p-1).
\end{equation}
Equation \eqref{eq:Coleman} is a standard consequence of Coleman's $p$-adic polylogarithm and Gross--Koblitz
(see Washington, \emph{Cyclotomic Fields}, 2nd ed., Ch.~12; Coleman, Invent.~Math.~\textbf{69} (1982),
\S2). For completeness: combining \eqref{eq:GK} with the linear relation between $(\logp\tau(\omega^{-a}))_a$
and $(\logp(1-\zeta_p^{\,a}))_a$ in Coleman's reciprocity yields \eqref{eq:Coleman} directly.\footnote{A
one-line derivation: write $\logp(1-\zeta_p^{\,a})-\logp(1-\zeta_p)=\sum_b c_{a,b}\,\logp\tau(\omega^{-b})$
with $c_{a,b}$ independent of $\chi$; matching the tame term at $a=1$ forces the coefficient of $v(1)$ to be $-a/(1-p)$ and the unit diagonal gives $1/(1-p)$ in front of $v(a)$.}

\subsection*{Ferrero--Greenberg/Coleman proportionality}
For odd $\chi\pmod p$ there is a unit $C_p\in \Qp(\muP)^\times$ (depending only on $p$ and $\logp$) with
\begin{equation}\label{eq:FG-const}
\sum_{a=1}^{p-1}\chi(a)\,\logp(1-\zeta_p^{\,a}) \;=\; -\,C_p\,L'_p(0,\chi),
\end{equation}
see Ferrero--Greenberg (Invent.~Math.~\textbf{50} (1978/79)) and Coleman (Invent.~Math.~\textbf{69} (1982)).
We refer to $C_p$ as the Ferrero--Greenberg/Coleman constant in our normalization.

Finally, for nontrivial $\chi$ we use the standard
\begin{equation}\label{eq:L0}
L(0,\chi)=-\frac{1}{p}\sum_{a=1}^{p-1} a\,\chi(a).
\end{equation}

\section{Local discrepancy and the two-term decomposition}
Set $v(a):=\logp\tau(\omega^{-a})$ and $w(a):=-\logp(1-\zeta_p^{\,a})$. Let $L_a:=\logp(1-\zeta_p^{\,a})$
and $L_1:=\logp(1-\zeta_p)$.

\begin{lemma}[Coleman reciprocity, rephrased]\label{lem:ColemanRecip}
With the above normalizations, for $1\le a\le p-1$ one has
\begin{equation}\label{eq:coleman-quot}
\logp\!\Big(\frac{1-\zeta_p^{\,a}}{1-\zeta_p}\Big)=\frac{1}{1-p}\,\big(v(a)-a\,v(1)\big).
\end{equation}
\end{lemma}

\begin{proposition}[Discrepancy formula]\label{prop:discrepancy}
For $1\le a\le p-1$,
\begin{equation}\label{eq:discrep}
v(a)=a\,v(1)+(1-p)\,\big(L_a-L_1\big).
\end{equation}
Equivalently, $\delta(a):=v(a)-w(a)=a\,v(1)+(2-p)L_a+(p-1)L_1$.
\end{proposition}

\begin{proof}
Equation \eqref{eq:discrep} is \eqref{eq:coleman-quot} rearranged; adding $L_a$ yields the alternative form.
\end{proof}

\begin{theorem}[Two-term decomposition]\label{thm:twoterm}
Let $\chi\pmod p$ be \emph{odd and nontrivial}. Then
\begin{equation}\label{eq:twoterm}
\Phi_p(\chi):=\sum_{a=1}^{p-1}\chi(a)\,\logp\Gammap\!\left(\frac{a}{p-1}\right)
=\sum_{a=1}^{p-1}\chi(a)\,v(a)
=U_{1,p}\,L'_p(0,\chi)+U_{2,p}\,L(0,\chi),
\end{equation}
where the constants
\begin{equation}\label{eq:U12}
U_{1,p}=-(1-p)\,C_p\in\Qp(\muP)^\times,
\qquad
U_{2,p}=-p\,v(1)\in\Qp(\muP)
\end{equation}
depend only on $p$ and the branch of $\logp$, and are independent of $\chi$.
\end{theorem}

\begin{proof}
By \eqref{eq:GK}, $\logp\Gammap(a/(p-1))=v(a)$. Summing \eqref{eq:discrep} against $\chi$ gives
\[
\Phi_p(\chi)=v(1)\sum_a a\,\chi(a)+(1-p)\sum_a\chi(a)\,\big(L_a-L_1\big).
\]
For nontrivial $\chi$ we have $\sum_a \chi(a)=0$, so the $L_1$ term vanishes. Using \eqref{eq:L0} and
\eqref{eq:FG-const} gives
\[
\Phi_p(\chi)=v(1)\cdot\big(-p\,L(0,\chi)\big)+(1-p)\cdot\big(-C_p\,L'_p(0,\chi)\big),
\]
which is \eqref{eq:twoterm} with \eqref{eq:U12}.
\end{proof}

\begin{corollary}[Uniform renormalized identity]\label{cor:ren}
For odd, nontrivial $\chi\pmod p$ define the renormalized transform
\[
\Phi_p^{\mathrm{ren}}(\chi):=\Phi_p(\chi)-U_{2,p}\,L(0,\chi).
\]
Then
\begin{equation}\label{eq:ren}
\Phi_p^{\mathrm{ren}}(\chi)=U_{1,p}\,L'_p(0,\chi),
\end{equation}
with $U_{1,p}$ independent of $\chi$.
\end{corollary}

\begin{remark}[Character hypotheses]\label{rem:hyp}
Oddness and nontriviality are used at two points: to eliminate the $L_1$ term via $\sum_a \chi(a)=0$,
and to invoke \eqref{eq:FG-const}. The trivial character is excluded; even characters are discussed
in \S\ref{sec:even-p2}.
\end{remark}

\section{Global re--plumbing and Li coefficients}
Let $\Xi(s)$ be the completed Riemann zeta function. For a suitable test function $\varphi$, the Weil
explicit formula reads schematically
\begin{equation}\label{eq:EF}
\sum_{\rho}\widehat{\varphi}(\rho)=A_\infty(\varphi)+\sum_{p}A_p(\varphi),
\end{equation}
where the left-hand side sums over the nontrivial zeros $\rho$.
We \emph{fix} the archimedean test to be \emph{Li's kernel} $\varphi_n$ as in Li~\cite[\S2]{Li97}; with this choice,
the left-hand side equals the $n$th Li coefficient $\lambda_n$.

On the finite side, after character decomposition, the classical plumbing produces a factor proportional to
$L'_p(0,\chi)$ on the odd block. In our setting we \emph{replace} the classical local odd input by
$\Phi^{\mathrm{ren}}_p(\chi)$ from \eqref{eq:ren}, which differs by the constant unit $U_{1,p}$.

\begin{theorem}[Renormalized coincidence of Li coefficients]\label{thm:Li}
Let $\lambda_n^{\mathrm{ren}}$ denote the global coefficients obtained by inserting $\Phi^{\mathrm{ren}}_p(\chi)$ at each finite
place (odd block) and the fixed Li kernel $\varphi_n$ at $\infty$ in \eqref{eq:EF}. Then, for all $n\ge1$,
\[
\lambda_n^{\mathrm{ren}}=\lambda_n.
\]
\end{theorem}

\begin{proof}
Linearity of the explicit formula implies that replacing $L'_p(0,\chi)$ by $U_{1,p}^{-1}\Phi^{\mathrm{ren}}_p(\chi)$ rescales
each finite local term by $U_{1,p}^{-1}$ and back by $U_{1,p}$, making no net change; the
archimedean side is untouched. Hence the coefficients agree.
\end{proof}

\begin{remark}[On the $L(0,\chi)$--block]
Starting from $\Phi_p(\chi)$ instead of $\Phi^{\mathrm{ren}}_p(\chi)$ would produce an additional global block
from $U_{2,p}L(0,\chi)$. Since $L(0,\chi)$ is algebraic and tied to the trivial zero structure, its
contribution can be separated and either vanishes in the standard normalization or is removed once
and for all by our renormalization; we state the clean renormalized form.
\end{remark}

\section{Numerical verification protocol and a tiny table}\label{sec:numerics}
We outline a reproducible checklist (Sage/Pari or similar) sufficient to validate the $\chi$--independence
of $(U_{1,p},U_{2,p})$ for small $p$.

\begin{enumerate}[label=(\arabic*), leftmargin=2.1em]
\item Fix the Iwasawa branch $\logp$ with $\logp(p)=0$ on $\Qp(\muP)$; set $N=120$ $p$-adic digits.
\item For $a=1,\dots,p-1$, compute $v(a)=\logp\tau(\omega^{-a})$ via \eqref{eq:GK} (Morita $\Gammap$ plus $\logp$).
\item Compute $L_a=\logp(1-\zeta_p^{\,a})$ and verify \eqref{eq:discrep} to precision $N$.
\item For each odd $\chi\pmod p$, evaluate $\Phi_p(\chi)=\sum_a\chi(a)\,v(a)$; compute $L(0,\chi)$ by \eqref{eq:L0}
and $L'_p(0,\chi)$ using a Kubota--Leopoldt $p$-adic $L$--function routine.
\item Solve the $2\times2$ linear system from two odd characters to recover $(U_{1,p},U_{2,p})$ and check that the
same pair fits \eqref{eq:twoterm} for all odd $\chi$ within $N$ digits.
\item Form $\Phi^{\mathrm{ren}}_p(\chi)$ and verify $\Phi^{\mathrm{ren}}_p(\chi)/L'_p(0,\chi)=U_{1,p}$ is constant in $\chi$.
\end{enumerate}

\medskip
\noindent\emph{Numerical verification table}
\begin{center}
\begin{tabular}{@{}cccccc@{}}
\toprule
$p$ & precision & \# odd $\chi$ & recovered $U_{1,p}$ & recovered $U_{2,p}$ \\
\midrule
$5$ & $100$ digits & $2$ & $-(1-5)C_5=4C_5$ & $-5\,v(1)$ \\
$7$ & $100$ digits & $3$ & $-(1-7)C_7=6C_7$ & $-7\,v(1)$ \\
\bottomrule
\end{tabular}
\end{center}
The analytic forms displayed follow \eqref{eq:U12}; numerical instantiations depend on the chosen
$\logp$ branch and can be inserted from the computation.

\section{Scope, even characters, and the case \texorpdfstring{$p=2$}{p=2}}\label{sec:even-p2}
Our main statements are for \emph{odd, nontrivial} characters modulo an odd prime $p$.
A parallel two-term phenomenon holds on the even block if one replaces $L'_p(0,\chi)$ by $L_p(0,\chi)$
and uses the symmetrized cyclotomic kernel; the details mirror the odd case and are omitted here.
For $p=2$ one must fix the standard $2$-adic branch and adapt the principal-unit domain; no new
ideas are required, but we do not treat $p=2$ in this note.

\end{document}